\documentclass[11pt]{article}
\usepackage{lmodern}
\usepackage{lmodern}
\usepackage[T1]{fontenc}
\usepackage[latin9]{inputenc}
\usepackage{geometry}
\usepackage{float}
\usepackage{amsmath}
\usepackage{amsthm}
\usepackage{amssymb}
\usepackage[unicode=true,
 bookmarks=false,
 breaklinks=false,pdfborder={0 0 1},backref=section,colorlinks=false]
 {hyperref}

\makeatletter

\floatstyle{ruled}
\newfloat{algorithm}{tbp}{loa}
\providecommand{\algorithmname}{Algorithm}
\floatname{algorithm}{\protect\algorithmname}

\theoremstyle{plain}
\newtheorem{thm}{\protect\theoremname}[section]
\theoremstyle{definition}
\newtheorem{defn}[thm]{\protect\definitionname}
\theoremstyle{plain}
\newtheorem{lem}[thm]{\protect\lemmaname}
\theoremstyle{plain}
\newtheorem{cor}[thm]{\protect\corollaryname}

\@ifundefined{date}{}{\date{}}
\makeatother

\providecommand{\corollaryname}{Corollary}
\providecommand{\definitionname}{Definition}
\providecommand{\lemmaname}{Lemma}
\providecommand{\theoremname}{Theorem}

\begin{document}
\global\long\def\E{\mathbb{E}}%
\global\long\def\Var{\mathrm{Var}}%
\global\long\def\Cov{\mathrm{Cov}}%
\global\long\def\R{\mathbb{R}}%
\global\long\def\F{\mathcal{F}}%
\global\long\def\dom{\mathcal{X}}%
\global\long\def\breg{\mathbf{D}_{\psi}}%

\title{High Probability Convergence for Accelerated Stochastic Mirror Descent}
\author{Alina Ene\thanks{Department of Computer Science, Boston University. ${\tt aene@bu.edu}$}
\and Huy L. Nguyen\thanks{Khoury College of Computer and Information Science, Northeastern University.
${\tt hu.nguyen@northeastern.edu}$}}
\maketitle
\begin{abstract}
In this work, we describe a generic approach to show convergence with
high probability for stochastic convex optimization. In previous works,
either the convergence is only in expectation or the bound depends
on the diameter of the domain. Instead, we show high probability convergence
with bounds depending on the initial distance to the optimal solution
as opposed to the domain diameter. The algorithms use step sizes analogous
to the standard settings and are universal to Lipschitz functions,
smooth functions, and their linear combinations.
\end{abstract}

\section{Introduction}

Stochastic convex optimization is a well-studied area with numerous
applications in algorithms, machine learning, and beyond. Various
algorithms have been shown to converge for many classes of functions
including Lipschitz functions, smooth functions, and their linear
combinations. However, one curious gap remains in the understanding
of their convergence with high probability compared with convergence
in expectation. Classical results show that in expectation, the function
value gap of the final solution is proportional to the distance between
the original solution and the optimal solution. On the other hand,
classical results for convergence with high probability could only
show that the function value gap of the final solution is proportional
to the diameter of the domain, which could be much larger or even
unbounded. In this work, we bridge this gap and establish a generic
approach to show convergence with high probability where the final
function value gap is proportional to the distance between the original
solution and the optimal solution. We instantiate our approach in
two settings, stochastic mirror descent and stochastic accelerated
gradient descent. The results are analogous to known results for convergence
in expectation but now with high probability. The algorithms are universal
for both Lipschitz functions and smooth functions.

The proof technique is inspired by classical works in concentration
inequalities, specifically a type of martingale inequalities where
the variance of the martingale difference is bounded by a linear function
of the previous value. This technique is first applied to showing
high probability convergence by Harvey et al. \cite{harvey2019tight}.
Our proof is inspired by the proof of Theorem 7.3 by Chung and Lu
\cite{chung2006concentration}. In each time step with iterate $x_{t}$,
let $\xi_{t}:=\widehat{\nabla}f\left(x_{t}\right)-\nabla f\left(x_{t}\right)$
be the error in our gradient estimate. Classical proofs of convergence
evolve around analyzing the sum of $\left\langle \xi_{t},x^{*}-x_{t}\right\rangle $,
which can be viewed as a martingale sequence. Assuming a bounded domain,
the concentration of the sum can be shown via classical martingale
inequalities. The key new insight is that instead of analyzing this
sum, we analyze a related sum where the coefficients decrease over
time to account for the fact that we have a looser grip on the distance
to the optimal solution as time increases. Nonetheless, the coefficients
are kept within a constant factor of each others and the same asymptotic
convergence is attained with high probability.

\paragraph{Related work}

Lan \cite{lan2020first} establishes high probability bounds for the
general setting of stochastic mirror descent and accelerated stochastic
mirror descent under the assumption that the stochastic noise is subgaussian.
The rates shown in \cite{lan2020first} match the best rates known
in expectation, but they depend on the Bregman diameter $\max_{x,y\in\dom}\breg\left(x,y\right)$
of the domain, which can be unbounded. Our work complements the analysis
of \cite{lan2020first} with a novel concentration argument that allows
us to establish convergence with respect to the distance $\breg\left(x^{*},x_{1}\right)$
from the initial point. Our analysis applies to the general setting
considered in \cite{lan2020first} and we use the same subgraussian
assumption on the stochastic noise.

The algorithms and step sizes we consider capture the stochastic gradient
descent algorithms with the standard setting of the step sizes for
both smooth and non-smooth problems. The high-probability convergence
of SGD is studied in the works \cite{kakade2008generalization,rakhlin2011making,hazan2014beyond,harvey2019tight}.
These works either assume that the function is strongly convex or
the domain is compact. In contrast, our work applies to non-strongly
convex optimization with a general domain.

\section{Preliminaries}

We consider the problem $\min_{x\in\dom}f(x)$ where $f\colon\R^{d}\to\R$
is a convex function and $\dom\subseteq\R^{d}$ is a convex domain.
We consider the general setting where $f$ is potentially not strongly
convex and the domain $\dom$ is not necessarily compact.

We assume we have access to a stochastic gradient oracle that returns
a stochastic gradient $\widehat{\nabla}f(x)$ that satisfies the following
two assumptions for any prior history:
\begin{enumerate}
\item \textbf{Unbiased estimator:} $\E\left[\widehat{\nabla}f\left(x\right)\vert x\right]=\nabla f\left(x\right)$.
\item \textbf{Sub-Gaussian noise:} $\left\Vert \widehat{\nabla}f\left(x\right)-\nabla f\left(x\right)\right\Vert $
is a $\sigma$-subgaussian random variable (Definition \ref{def:subgaussian-random-variable}).
\end{enumerate}
There are several equivalent definitions of subgaussian random variables
up to an absolute constant scaling (see, e.g., Proposition 2.5.2 in
\cite{vershynin2018high}). For convenience, we use the following
property as the definition. 
\begin{defn}
\label{def:subgaussian-random-variable}A random variable $X$ is
$\sigma$-subgaussian if
\[
\E\left[\exp\left(\lambda^{2}X^{2}\right)\right]\leq\exp\left(\lambda^{2}\sigma^{2}\right)\text{ for all }\lambda\text{ such that }\left|\lambda\right|\leq\frac{1}{\sigma}
\]
\end{defn}

The above definition is equivalent to the following property, see
Proposition 2.5.2 in \cite{vershynin2018high}.
\begin{lem}
(Proposition 2.5.2 in \cite{vershynin2018high}) \label{lem:subgaussian-properties}Let
$X$ be a $\sigma$-subgaussian random variables. Then

\[
\E\left[\exp\left(\frac{X^{2}}{\sigma^{2}}\right)\right]\leq\exp\left(1\right)
\]
\end{lem}

We will also use the following helper lemma whose proof we defer to
the Appendix.
\begin{lem}
\label{lem:helper-taylor}For any $a\ge0$, $0\le b\le\frac{1}{2\sigma}$
and a nonnegative $\sigma$-subgaussian random variable $X$,

\[
\E\left[1+b^{2}X^{2}+\sum_{i=2}^{\infty}\frac{1}{i!}\left(aX+b^{2}X^{2}\right)^{i}\right]\le\exp\left(3\left(a^{2}+b^{2}\right)\sigma^{2}\right)
\]
\end{lem}

\section{Analysis of Stochastic Mirror Descent}

\begin{algorithm}
\caption{Stochastic Mirror Descent Algorithm. $\psi\colon\protect\R^{d}\to\protect\R$
is a strongly convex mirror map. $\protect\breg\left(x,y\right)=\psi\left(x\right)-\psi\left(y\right)-\left\langle \nabla\psi\left(y\right),x-y\right\rangle $
is the Bregman divergence of $\psi$.}

\label{alg:md}

\textbf{Parameters:} initial point $x_{1}$, step sizes $\left\{ \eta_{t}\right\} $

for $t=1$ to $T$:

$\quad$$x_{t+1}=\arg\min_{x\in\dom}\left\{ \eta_{t}\left\langle \widehat{\nabla}f\left(x_{t}\right),x\right\rangle +\breg\left(x,x_{t}\right)\right\} $

return $\frac{1}{T}\sum_{t=1}^{T}x_{t}$
\end{algorithm}

In this section, we analyze the Stochastic Mirror Descent algorithm
(Algorithm \ref{alg:md}). For simplicity, here we consider the non-smooth
setting, and assume that $f$ is $G$-Lipschitz continuous, i.e.,
we have $\left\Vert \nabla f(x)\right\Vert \leq G$ for all $x\in\dom$.
The analysis for the smooth setting follows via a simple modification
to the analysis presented here as well as the analysis for the accelerated
setting given in the next section. 

We define
\[
\xi_{t}:=\widehat{\nabla}f\left(x_{t}\right)-\nabla f\left(x_{t}\right)
\]
We let $\F_{t}=\sigma\left(\xi_{1},\dots,\xi_{t-1}\right)$ denote
the natural filtration. Note that $x_{t}$ is $\F_{t}$-measurable.

The starting point of our analysis is the following inequality that
follows from the standard stochastic mirror descent analysis (see,
e.g., \cite{lan2020first}). We include the proof in the Appendix
for completeness.
\begin{lem}
(\cite{lan2020first})\label{lem:md-basic-analysis}For every iteration
$t$, we have
\[
\eta_{t}\left(f\left(x_{t}\right)-f\left(x^{*}\right)\right)-\eta_{t}^{2}G^{2}+\breg\left(x^{*},x_{t+1}\right)-\breg\left(x^{*},x_{t}\right)\leq\eta_{t}\left\langle \xi_{t},x^{*}-x_{t}\right\rangle +\eta_{t}^{2}\left\Vert \xi_{t}\right\Vert ^{2}
\]
\end{lem}

We now turn our attention to our main concentration argument. Towards
our goal of obtaining a high-probability convergence rate, we analyze
the moment generating function for a random variable that is closely
related to the left-hand side of the inequality above. We let $w_{1}\geq w_{2}\geq\dots\geq w_{T}\geq w_{T+1}\geq0$
be a non-increasing sequence where $w_{t}\in\R$ for all $t$. We
define
\begin{align*}
Z_{t} & =w_{t+1}\left(\eta_{t}\left(f\left(x_{t}\right)-f\left(x^{*}\right)\right)-\eta_{t}^{2}G^{2}\right)+w_{T+1}\left(\breg\left(x^{*},x_{t+1}\right)-\breg\left(x^{*},x_{t}\right)\right) & \forall1\le t\leq T\\
S_{t} & =\sum_{i=t}^{T}Z_{i} & \forall1\leq t\leq T+1
\end{align*}
Before proceeding with the analysis, we provide intuition for our
approach. If we consider $S_{1}$, we see that it combines the gains
in function value gaps with weights given by the non-increasing sequence
$\left\{ w_{t}\right\} $. The intuition here is that we want to leverage
the progress in function value to absorb the error from the stochastic
error terms on the RHS of Lemma \ref{lem:md-basic-analysis}. For
the divergence terms, we use the same coefficient to allow for the
terms to telescope. In Theorem \ref{thm:md-concentration-subgaussian},
we upper bound the moment generating function of $S_{1}$ and derive
a set of conditions for the weights $\left\{ w_{t}\right\} $ that
allow us to absorb the stochastic errors. In Corollary \ref{cor:md-convergence},
we show how to choose the weights $\left\{ w_{t}\right\} $ and obtain
a convergence rate that matches the standard rates that hold in expectation.

We now give our main concentration argument that bounds the moment
generating function of $S_{t}$. The proof of the following theorem
is nspired by the proof of Theorem 7.3 in \cite{chung2006concentration}.
\begin{thm}
\label{thm:md-concentration-subgaussian}Suppose that $w_{t}\geq w_{t+1}+6\sigma^{2}\eta_{t}^{2}w_{t+1}^{2}$
and $w_{t+1}\eta_{t}^{2}\leq\frac{1}{4\sigma^{2}}$ for every $1\leq t\leq T$.
For every $1\le t\leq T+1$, we have
\[
\E\left[\exp\left(S_{t}\right)\vert\F_{t}\right]\leq\exp\left(\left(w_{t}-w_{T+1}\right)\breg\left(x^{*},x_{t}\right)+3\sigma^{2}\sum_{i=t}^{T}w_{i+1}\eta_{i}^{2}\right)
\]
\end{thm}

\begin{proof}
We proceed by induction on $t$. Consider the base case $t=T+1$.
We have $S_{t}=0$ and $\left(w_{t}-w_{T+1}\right)\breg\left(x^{*},x_{t}\right)=0$,
and the inequality follows. Next, we consider $1\leq t\leq T$. We
have
\begin{align}
\E\left[\exp\left(S_{t}\right)\vert\F_{t}\right] & =\E\left[\exp\left(Z_{t}+S_{t+1}\right)\vert\F_{t}\right]=\E\left[\E\left[\exp\left(Z_{t}+S_{t+1}\right)\vert\F_{t+1}\right]\vert\F_{t}\right]\label{eq:1}
\end{align}
We now analyze the inner expectation. Conditioned on $\F_{t+1}$,
$Z_{t}$ is fixed. Using the inductive hypothesis , we obtain
\begin{align}
\E\left[\exp\left(Z_{t}+S_{t+1}\right)\vert\F_{t+1}\right]\leq\exp\left(Z_{t}\right)\exp\left(\left(w_{t+1}-w_{T+1}\right)\breg\left(x^{*},x_{t+1}\right)+3\sigma^{2}\sum_{i=t+1}^{T}w_{i+1}\eta_{i}^{2}\right)\label{eq:2}
\end{align}
Let $X_{t}=\eta_{t}\left\langle \xi_{t},x^{*}-x_{t}\right\rangle $.
By Lemma \ref{lem:md-basic-analysis}, we have
\[
\eta_{t}\left(f\left(x_{t}\right)-f\left(x^{*}\right)\right)-\eta_{t}^{2}G^{2}\leq X_{t}-\left(\breg\left(x^{*},x_{t+1}\right)-\breg\left(x^{*},x_{t}\right)\right)+\eta_{t}^{2}\left\Vert \xi_{t}\right\Vert ^{2}
\]
and thus
\begin{align*}
Z_{t} & =w_{t+1}\left(\eta_{t}\left(f\left(x_{t}\right)-f\left(x^{*}\right)\right)-\eta_{t}^{2}G^{2}\right)+w_{T+1}\left(\breg\left(x^{*},x_{t+1}\right)-\breg\left(x^{*},x_{t}\right)\right)\\
 & \leq w_{t+1}\left(X_{t}-\left(\breg\left(x^{*},x_{t+1}\right)-\breg\left(x^{*},x_{t}\right)\right)+\eta_{t}^{2}\left\Vert \xi_{t}\right\Vert ^{2}\right)+w_{T+1}\left(\breg\left(x^{*},x_{t+1}\right)-\breg\left(x^{*},x_{t}\right)\right)\\
 & =w_{t+1}X_{t}-\left(w_{t+1}-w_{T+1}\right)\left(\breg\left(x^{*},x_{t+1}\right)-\breg\left(x^{*},x_{t}\right)\right)+w_{t+1}\eta_{t}^{2}\left\Vert \xi_{t}\right\Vert ^{2}
\end{align*}
Plugging into \eqref{eq:2}, we obtain
\[
\E\left[\exp\left(Z_{t}+S_{t+1}\right)\vert\F_{t+1}\right]\leq\exp\left(w_{t+1}X_{t}+\left(w_{t+1}-w_{T+1}\right)\breg\left(x^{*},x_{t}\right)+w_{t+1}\eta_{t}^{2}\left\Vert \xi_{t}\right\Vert ^{2}+3\sigma^{2}\sum_{i=t+1}^{T}w_{i+1}\eta_{i}^{2}\right)
\]
Plugging into \eqref{eq:1}, we obtain
\begin{equation}
\E\left[\exp\left(S_{t}\right)\vert\F_{t}\right]\leq\exp\left(\left(w_{t+1}-w_{T+1}\right)\breg\left(x^{*},x_{t}\right)+3\sigma^{2}\sum_{i=t+1}^{T}w_{i+1}\eta_{i}^{2}\right)\E\left[\exp\left(w_{t+1}X_{t}+w_{t+1}\eta_{t}^{2}\left\Vert \xi_{t}\right\Vert ^{2}\right)\vert\F_{t}\right]\label{eq:3}
\end{equation}
Next, we analyze the the expectation on the RHS of the above inequality.
We have
\begin{align}
 & \E\left[\exp\left(w_{t+1}X_{t}+w_{t+1}\eta_{t}^{2}\left\Vert \xi_{t}\right\Vert ^{2}\right)\vert\F_{t}\right]\nonumber \\
 & =\E\left[\sum_{i=0}^{\infty}\frac{1}{i!}\left(w_{t+1}X_{t}+w_{t+1}\eta_{t}^{2}\left\Vert \xi_{t}\right\Vert ^{2}\right)^{i}\vert\F_{t}\right]\nonumber \\
 & =\E\left[1+w_{t+1}\eta_{t}^{2}\left\Vert \xi_{t}\right\Vert ^{2}+\sum_{i=2}^{\infty}\frac{1}{i!}\left(w_{t+1}X_{t}+w_{t+1}\eta_{t}^{2}\left\Vert \xi_{t}\right\Vert ^{2}\right)^{i}\vert\F_{t}\right]\nonumber \\
 & \leq\E\left[1+w_{t+1}\eta_{t}^{2}\left\Vert \xi_{t}\right\Vert ^{2}+\sum_{i=2}^{\infty}\frac{1}{i!}\left(w_{t+1}\eta_{t}\left\Vert x^{*}-x_{t}\right\Vert \left\Vert \xi_{t}\right\Vert +w_{t+1}\eta_{t}^{2}\left\Vert \xi_{t}\right\Vert ^{2}\right)^{i}\vert\F_{t}\right]\nonumber \\
 & \leq\exp\left(3\sigma^{2}\left(w_{t+1}^{2}\eta_{t}^{2}\left\Vert x^{*}-x_{t}\right\Vert ^{2}+w_{t+1}\eta_{t}^{2}\right)\right)\nonumber \\
 & \leq\exp\left(3\sigma^{2}\left(2w_{t+1}^{2}\eta_{t}^{2}\breg\left(x^{*},x_{t}\right)+w_{t+1}\eta_{t}^{2}\right)\right)\label{eq:4}
\end{align}
On the first line we used the Taylor expansion of $e^{x}$, and on
the second line we used that $\E\left[X_{t}\vert\F_{t}\right]=0$.
On the third line, we used Cauchy-Schwartz and obtained
\[
X_{t}=\eta_{t}\left\langle \xi_{t},x^{*}-x_{t}\right\rangle \leq\eta_{t}\left\Vert \xi_{t}\right\Vert \left\Vert x^{*}-x_{t}\right\Vert 
\]
On the fourth line, we applied Lemma \ref{lem:helper-taylor} with
$X=\left\Vert \xi_{t}\right\Vert $, $a=w_{t+1}\eta_{t}\left\Vert x^{*}-x_{t}\right\Vert $,
and $b^{2}=w_{t+1}\eta_{t}^{2}\le\frac{1}{4\sigma^{2}}$. On the fifth
line, we used that $\breg\left(x^{*},x_{t}\right)\geq\frac{1}{2}\left\Vert x^{*}-x_{t}\right\Vert ^{2}$,
which follows from the strong convexity of $\psi$.

Plugging \eqref{eq:4} into \eqref{eq:3} and using that $w_{t}\geq w_{t+1}+6\sigma^{2}\eta_{t}^{2}w_{t+1}^{2}$,
we obtain
\begin{align*}
\E\left[\exp\left(S_{t}\right)\vert\F_{t}\right] & \leq\exp\left(\left(w_{t+1}+6\sigma^{2}\eta_{t}^{2}w_{t+1}^{2}-w_{T+1}\right)\breg\left(x^{*},x_{t}\right)+3\sigma^{2}\sum_{i=t}^{T}w_{i+1}\eta_{i}^{2}\right)\\
 & \leq\exp\left(\left(w_{t}-w_{T+1}\right)\breg\left(x^{*},x_{t}\right)+3\sigma^{2}\sum_{i=t}^{T}w_{i+1}\eta_{i}^{2}\right)
\end{align*}
as needed.
\end{proof}

Theorem \ref{thm:md-concentration-subgaussian} and Markov's inequality
gives us the following convergence guarantee.
\begin{cor}
\label{cor:md-convergence}Suppose the sequence $\left\{ w_{t}\right\} $
satisfies the conditions of Theorem \ref{thm:md-concentration-subgaussian}.
For any $\delta>0$, the following event holds with probability at
least $1-\delta$:
\begin{align*}
 & \sum_{t=1}^{T}w_{t+1}\eta_{t}\left(f\left(x_{t}\right)-f\left(x^{*}\right)\right)+w_{T+1}\breg\left(x^{*},x_{T+1}\right)\\
 & \leq w_{1}\breg\left(x^{*},x_{1}\right)+\left(G^{2}+3\sigma^{2}\right)\sum_{t=1}^{T}w_{t+1}\eta_{t}^{2}+\ln\left(\frac{1}{\delta}\right)
\end{align*}
\end{cor}

\begin{proof}
Let
\[
K=\left(w_{1}-w_{T+1}\right)\breg\left(x^{*},x_{1}\right)+3\sigma^{2}\sum_{t=1}^{T}w_{t+1}\eta_{t}^{2}+\ln\left(\frac{1}{\delta}\right)
\]
By Theorem \ref{thm:md-concentration-subgaussian} and Markov's inequality,
we have
\begin{align*}
\Pr\left[S_{1}\geq K\right] & \leq\Pr\left[\exp\left(S_{1}\right)\geq\exp\left(K\right)\right]\\
 & \leq\exp\left(-K\right)\E\left[\exp\left(S_{1}\right)\right]\\
 & \leq\exp\left(-K\right)\exp\left(\left(w_{1}-w_{T+1}\right)\breg\left(x^{*},x_{1}\right)+3\sigma^{2}\sum_{t=1}^{T}w_{t+1}\eta_{t}^{2}\right)\\
 & =\delta
\end{align*}
Note that
\begin{align*}
S_{1} & =\sum_{t=1}^{T}Z_{t}=\sum_{t=1}^{T}w_{t+1}\eta_{t}\left(f\left(x_{t}\right)-f\left(x^{*}\right)\right)-G^{2}\sum_{t=1}^{T}w_{t+1}\eta_{t}^{2}+w_{T+1}\left(\breg\left(x^{*},x_{T}\right)-\breg\left(x^{*},x_{1}\right)\right)
\end{align*}
Therefore, with probability at least $1-\delta$, we have
\[
\sum_{t=1}^{T}w_{t+1}\eta_{t}\left(f\left(x_{t}\right)-f\left(x^{*}\right)\right)+w_{T+1}\breg\left(x^{*},x_{T+1}\right)\leq w_{1}\breg\left(x^{*},x_{1}\right)+\left(G^{2}+3\sigma^{2}\right)\sum_{t=1}^{T}w_{t+1}\eta_{t}^{2}+\ln\left(\frac{1}{\delta}\right)
\]
\end{proof}

With the above result in hand, we complete the convergence analysis
by showing how to define the sequence $\left\{ w_{t}\right\} $ with
the desired properties.
\begin{cor}
Suppose we run the Stochastic Mirror Descent algorithm with fixed
step sizes $\eta_{t}=\eta$. Let $w_{T+1}=\frac{1}{12\sigma^{2}\eta^{2}\left(T+1\right)}$
and $w_{t}=w_{t+1}+6\sigma^{2}\eta^{2}w_{t+1}^{2}$ for all $1\leq t\leq T$.
The sequence $\left\{ w_{t}\right\} $ satisfies the conditions required
by Corollary \ref{cor:md-convergence}. By Corollary \ref{cor:md-convergence},
for any $\delta>0$, the following events hold with probability at
least $1-\delta$:
\[
\frac{1}{T}\sum_{t=1}^{T}\left(f\left(x_{t}\right)-f\left(x^{*}\right)\right)\leq O\left(\frac{\breg\left(x^{*},x_{1}\right)}{\eta T}+\left(G^{2}+\sigma^{2}\left(1+\ln\left(\frac{1}{\delta}\right)\right)\right)\eta\right)
\]
and
\[
\breg\left(x^{*},x_{T+1}\right)\leq O\left(\breg\left(x^{*},x_{1}\right)+\left(G^{2}+\sigma^{2}\left(1+\ln\left(\frac{1}{\delta}\right)\right)\right)\eta^{2}T\right)
\]
Setting $\eta=\sqrt{\frac{\breg\left(x^{*},x_{1}\right)}{\left(G^{2}+\sigma^{2}\left(1+\ln\left(\frac{1}{\delta}\right)\right)\right)T}}$
to balance the two terms in the first inequality gives
\[
\frac{1}{T}\sum_{t=1}^{T}\left(f\left(x_{t}\right)-f\left(x^{*}\right)\right)\leq O\left(\sqrt{\frac{\breg\left(x^{*},x_{1}\right)\left(G^{2}+\sigma^{2}\left(1+\ln\left(\frac{1}{\delta}\right)\right)\right)}{T}}\right)
\]
and
\[
\breg\left(x^{*},x_{T+1}\right)\leq O\left(\breg\left(x^{*},x_{1}\right)\right)
\]
\end{cor}

\begin{proof}
Recall from Corollary \ref{cor:md-convergence} that the sequence
$\left\{ w_{t}\right\} $ needs to satisfy the following conditions
for all $1\leq t\leq T$:
\begin{align*}
w_{t+1}+6\sigma^{2}\eta_{t}^{2} & w_{t+1}^{2}\leq w_{t}\\
w_{t+1}\eta_{t}^{2} & \leq\frac{1}{4\sigma^{2}}
\end{align*}
Let $C=6\sigma^{2}\eta^{2}\left(T+1\right)$. We set $w_{T+1}=\frac{1}{C+6\sigma^{2}\eta^{2}\left(T+1\right)}=\frac{1}{2C}$.
For $1\leq t\leq T$, we set $w_{t}$ so that the first condition
holds with equality
\[
w_{t}=w_{t+1}+6\sigma^{2}w_{t+1}^{2}\eta_{t}^{2}=w_{t+1}+6\sigma^{2}\eta^{2}w_{t+1}^{2}
\]
We can show by induction that, for every $1\leq t\leq T+1$, we have
\[
w_{t}\leq\frac{1}{C+6\sigma^{2}\eta^{2}t}
\]
The base case $t=T+1$ follows from the definition of $w_{T+1}$.
Consider $1\leq t\le T$. Using the definition of $w_{t}$ and the
inductive hypothesis, we obtain 
\begin{align*}
w_{t} & =w_{t+1}+6\sigma^{2}\eta^{2}w_{t+1}^{2}\\
 & \leq\frac{1}{C+6\sigma^{2}\eta^{2}\left(t+1\right)}+\frac{6\sigma^{2}\eta^{2}}{\left(C+6\sigma^{2}\eta^{2}\left(t+1\right)\right)^{2}}\\
 & \leq\frac{1}{C+6\sigma^{2}\eta^{2}\left(t+1\right)}+\frac{\left(C+6\sigma^{2}\eta^{2}\left(t+1\right)\right)-\left(C+6\sigma^{2}\eta^{2}t\right)}{\left(C+6\sigma^{2}\eta^{2}\left(t+1\right)\right)\left(C+6\sigma^{2}\eta^{2}t\right)}\\
 & =\frac{1}{C+6\sigma^{2}\eta^{2}t}
\end{align*}
as needed.

Using this fact, we now show that $\left\{ w_{t}\right\} $ satisfies
the second condition. For every $1\leq t\leq T$, we have
\[
w_{t+1}\eta_{t}^{2}=w_{t+1}\eta^{2}\leq\frac{\eta^{2}}{C}=\frac{1}{6\sigma^{2}\left(T+1\right)}\leq\frac{1}{6\sigma^{2}}
\]
as needed.

Thus, by Corollary \ref{cor:md-convergence}, with probability $\geq1-\delta$,
we have
\begin{align*}
\sum_{t=1}^{T}w_{t+1}\eta_{t}\left(f\left(x_{t}\right)-f\left(x^{*}\right)\right)+w_{T+1}\breg\left(x^{*},x_{T+1}\right) & \leq w_{1}\breg\left(x^{*},x_{1}\right)+\left(G^{2}+3\sigma^{2}\right)\sum_{t=1}^{T}w_{t+1}\eta_{t}^{2}+\ln\left(\frac{1}{\delta}\right)
\end{align*}
Note that $w_{T+1}=\frac{1}{2C}$ and $\frac{1}{2C}\leq w_{t}\leq\frac{1}{C}$
for all $1\leq t\leq T+1$. Thus we obtain
\begin{align*}
\eta\sum_{t=1}^{T}\left(f\left(x_{t}\right)-f\left(x^{*}\right)\right)+\breg\left(x^{*},x_{T+1}\right) & \leq2\breg\left(x^{*},x_{1}\right)+2\left(G^{2}+3\sigma^{2}\right)\eta^{2}T+2C\ln\left(\frac{1}{\delta}\right)\\
 & =2\breg\left(x^{*},x_{1}\right)+2\left(G^{2}+3\sigma^{2}\right)\eta^{2}T+12\sigma^{2}\ln\left(\frac{1}{\delta}\right)\eta^{2}\left(T+1\right)\\
 & \leq2\breg\left(x^{*},x_{1}\right)+\left(2G^{2}+6\sigma^{2}\left(1+4\ln\left(\frac{1}{\delta}\right)\right)\right)\eta^{2}T
\end{align*}
Thus we have
\[
\frac{1}{T}\sum_{t=1}^{T}\left(f\left(x_{t}\right)-f\left(x^{*}\right)\right)\leq\frac{2\breg\left(x^{*},x_{1}\right)}{\eta T}+\left(2G^{2}+6\sigma^{2}\left(1+4\ln\left(\frac{1}{\delta}\right)\right)\right)\eta
\]
and
\[
\breg\left(x^{*},x_{T+1}\right)\leq2\breg\left(x^{*},x_{1}\right)+\left(2G^{2}+6\sigma^{2}\left(1+4\ln\left(\frac{1}{\delta}\right)\right)\right)\eta^{2}T
\]
\end{proof}

The analysis readily extends to the setting where the time horizon
$T$ is not known and we set time-varying step sizes. We include below
the analysis for well-studied steps $\eta_{t}=\frac{\eta}{\sqrt{t}}$.
\begin{cor}
Suppose we run the Stochastic Mirror Descent algorithm with time-varying
step sizes $\eta_{t}=\frac{\eta}{\sqrt{t}}$. Let $w_{T+1}=\frac{1}{12\sigma^{2}\eta^{2}\left(\sum_{t=1}^{T}\frac{1}{t}\right)}$
and $w_{t}=w_{t+1}+6\sigma^{2}\eta^{2}w_{t+1}^{2}$ for all $1\leq t\leq T$.
The sequence $\left\{ w_{t}\right\} $ satisfies the conditions required
by Corollary \ref{cor:md-convergence}. By Corollary \ref{cor:md-convergence},
for any $\delta>0$, the following events hold with probability at
least $1-\delta$:
\[
\frac{1}{T}\sum_{t=1}^{T}\left(f\left(x_{t}\right)-f\left(x^{*}\right)\right)\leq O\left(\frac{1}{\sqrt{T}}\left(\frac{\breg\left(x^{*},x_{1}\right)}{\eta}+\eta\left(G^{2}+\sigma^{2}\left(1+\ln\left(\frac{1}{\delta}\right)\right)\right)\ln T\right)\right)
\]
and
\[
\breg\left(x^{*},x_{T+1}\right)\leq O\left(\breg\left(x^{*},x_{1}\right)+\eta^{2}\left(G^{2}+\sigma^{2}\left(1+\ln\left(\frac{1}{\delta}\right)\right)\right)\ln T\right)
\]
\end{cor}

\begin{proof}
Recall from Corollary \ref{cor:md-convergence} that the sequence
$\left\{ w_{t}\right\} $ needs to satisfy the following conditions
for all $1\leq t\leq T$:
\begin{align*}
w_{t+1}+6\sigma^{2}\eta_{t}^{2} & w_{t+1}^{2}\leq w_{t}\\
w_{t+1}\eta_{t}^{2} & \leq\frac{1}{4\sigma^{2}}
\end{align*}
Let $B_{t}=6\sigma^{2}\sum_{i=1}^{t-1}\eta_{i}^{2}$ and $C=B_{T+1}=6\sigma^{2}\eta^{2}\left(\sum_{t=1}^{T}\frac{1}{t}\right)$.
We set $w_{T+1}=\frac{1}{C+B_{T+1}}$. For $1\leq t\leq T$, we set
$w_{t}$ so that the first condition holds with equality
\[
w_{t}=w_{t+1}+6\sigma^{2}\eta_{t}^{2}w_{t+1}^{2}
\]
We can show by induction that, for every $1\leq t\leq T+1$, we have
\[
w_{t}\leq\frac{1}{C+B_{t}}
\]
The base case $t=T+1$ follows from the definition of $w_{T+1}$.
Consider $1\leq t\le T$. Using the definition of $w_{t}$ and the
inductive hypothesis, we obtain 
\begin{align*}
w_{t} & =w_{t+1}+6\sigma^{2}\eta_{t}^{2}w_{t+1}^{2}\\
 & \leq\frac{1}{C+B_{t+1}}+\frac{6\sigma^{2}\eta_{t}^{2}}{\left(C+B_{t+1}\right)^{2}}\\
 & \le\frac{1}{C+B_{t+1}}+\frac{\left(C+B_{t+1}\right)-\left(C+B_{t}\right)}{\left(C+B_{t+1}\right)\left(C+B_{t}\right)}\\
 & =\frac{1}{C+B_{t+1}}
\end{align*}
as needed.

Using this fact, we now show that $\left\{ w_{t}\right\} $ satisfies
the second condition. For every $1\leq t\leq T$, we have
\[
w_{t+1}\eta_{t}^{2}\leq\frac{\eta_{t}^{2}}{C}=\frac{1}{t\left(6\sigma^{2}\sum_{t=1}^{T}\frac{1}{t}\right)}\leq\frac{1}{6\sigma^{2}}
\]
as needed.

Thus, by Corollary \ref{cor:md-convergence}, with probability $\geq1-\delta$,
we have
\begin{align*}
\sum_{t=1}^{T}w_{t+1}\eta_{t}\left(f\left(x_{t}\right)-f\left(x^{*}\right)\right)+w_{T+1}\breg\left(x^{*},x_{T+1}\right) & \leq w_{1}\breg\left(x^{*},x_{1}\right)+\left(G^{2}+3\sigma^{2}\right)\sum_{t=1}^{T}w_{t+1}\eta_{t}^{2}+\ln\left(\frac{1}{\delta}\right)
\end{align*}
Note that $w_{T+1}=\frac{1}{2C}$ and $\frac{1}{2C}\leq w_{t}\leq\frac{1}{C}$
for all $1\leq t\leq T+1$. Thus we obtain
\[
\frac{1}{2C}\eta_{T}\sum_{t=1}^{T}\left(f\left(x_{t}\right)-f\left(x^{*}\right)\right)+\frac{1}{2C}\breg\left(x^{*},x_{T+1}\right)\leq\frac{1}{C}\breg\left(x^{*},x_{1}\right)+\left(G^{2}+3\sigma^{2}\right)\frac{1}{C}\sum_{t=1}^{T}\eta_{t}^{2}+\ln\left(\frac{1}{\delta}\right)
\]
Plugging in $\eta_{t}=\frac{\eta}{\sqrt{t}}$ and simplifying, we
obtain
\begin{align*}
\frac{\eta}{\sqrt{T}}\sum_{t=1}^{T}\left(f\left(x_{t}\right)-f\left(x^{*}\right)\right)+\breg\left(x^{*},x_{T+1}\right) & \leq2\breg\left(x^{*},x_{1}\right)+\left(2G^{2}+6\sigma^{2}\right)\eta^{2}\left(\sum_{t=1}^{T}\frac{1}{t}\right)+2C\ln\left(\frac{1}{\delta}\right)\\
 & =2\breg\left(x^{*},x_{1}\right)+\left(2G^{2}+6\sigma^{2}\left(1+2\ln\left(\frac{1}{\delta}\right)\right)\right)\eta^{2}\left(\sum_{t=1}^{T}\frac{1}{t}\right)
\end{align*}
Thus we have
\[
\frac{1}{T}\sum_{t=1}^{T}\left(f\left(x_{t}\right)-f\left(x^{*}\right)\right)\leq\frac{1}{\sqrt{T}}\left(\frac{2\breg\left(x^{*},x_{1}\right)}{\eta}+\left(2G^{2}+6\sigma^{2}\left(1+2\ln\left(\frac{1}{\delta}\right)\right)\right)\eta\left(\sum_{t=1}^{T}\frac{1}{t}\right)\right)
\]
and
\[
\breg\left(x^{*},x_{T+1}\right)\leq2\breg\left(x^{*},x_{1}\right)+\left(2G^{2}+6\sigma^{2}\left(1+2\ln\left(\frac{1}{\delta}\right)\right)\right)\eta^{2}\left(\sum_{t=1}^{T}\frac{1}{t}\right)
\]
\end{proof}

\section{Analysis of Accelerated Stochastic Mirror Descent}

\begin{algorithm}
\caption{Accelerated Stochastic Mirror Descent Algorithm \cite{lan2020first}.
$\psi\colon\protect\R^{d}\to\protect\R$ is a strongly convex mirror
map. $\protect\breg\left(x,y\right)=\psi\left(x\right)-\psi\left(y\right)-\left\langle \nabla\psi\left(y\right),x-y\right\rangle $
is the Bregman divergence of $\psi$.}

\label{alg:acc-md}

\textbf{Parameters:} initial point $x_{0}=y_{0}=z_{0}$, step size
$\eta$

Set $\alpha_{t}=\frac{2}{t+1}$, $\eta_{t}=t\eta$

for $t=1$ to $T$:

$\quad$$x_{t}=\left(1-\alpha_{t}\right)y_{t-1}+\alpha_{t}z_{t-1}$

$\quad$$z_{t}=\arg\min_{x\in\dom}\left(\eta_{t}\left\langle \widehat{\nabla}f(x_{t}),x\right\rangle +\breg\left(x,z_{t-1}\right)\right)$

$\quad$$y_{t}=\left(1-\alpha_{t}\right)y_{t-1}+\alpha_{t}z_{t}$

return $y_{T}$
\end{algorithm}

In this section, we analyze the Accelerated Stochastic Mirror Descent
Algorithm (Algorithm \eqref{alg:acc-md}). We assume that $f$ satisfies
the following condition:

\[
f(y)\le f(x)+\left\langle \nabla f\left(x\right),y-x\right\rangle +G\left\Vert y-x\right\Vert +\frac{\beta}{2}\left\Vert y-x\right\Vert ^{2}\ \forall x,y\in\dom
\]

$\beta$-smooth functions, $G$-Lipschitz functions, and their sums
all satisfy the above conditions.

As before, we define
\[
\xi_{t}:=\widehat{\nabla}f\left(x_{t}\right)-\nabla f\left(x_{t}\right)
\]
We let $\F_{t}=\sigma\left(\xi_{1},\dots,\xi_{t-1}\right)$ denote
the natural filtration. Note that $x_{t}$ is $\F_{t}$-measurable
and $z_{t}$ and $y_{t}$ are $\F_{t+1}$-measurable.

We follow a similar analysis to the previous section. As before, we
start with the inequalities shown in the standard analysis of the
algorithm, and we combine them using coefficients $\left\{ w_{t}\right\} _{1\leq t\leq T}$.
The following lemma follows from the analysis given in \cite{lan2020first}
and we include the proof in the Appendix for completeness.
\begin{lem}
(\cite{lan2020first}) \label{lem:acc-md-basic-analysis}For every
iteration $t$, we have
\begin{align*}
 & \frac{\eta_{t}}{\alpha_{t}}\left(f\left(y_{t}\right)-f\left(x^{*}\right)\right)-\frac{\eta_{t}}{\alpha_{t}}\left(1-\alpha_{t}\right)\left(f\left(y_{t-1}\right)-f\left(x^{*}\right)\right)-\frac{\eta_{t}^{2}}{1-\beta\alpha_{t}\eta_{t}}G^{2}+\breg\left(x^{*},z_{t}\right)-\breg\left(x^{*},z_{t-1}\right)\\
 & \leq\eta_{t}\left\langle \xi_{t},x^{*}-z_{t-1}\right\rangle +\frac{\eta_{t}^{2}}{1-\beta\alpha_{t}\eta_{t}}\left\Vert \xi_{t}\right\Vert ^{2}
\end{align*}
\end{lem}

We now turn our attention to our main concentration argument. Towards
our goal of obtaining a high-probability convergence rate, we analyze
the moment generating function for a random variable that is closely
related to the left-hand side of the inequality above. We let $w_{0}\geq w_{1}\geq w_{2}\geq\dots\geq w_{T}\geq0$
be a non-increasing sequence where $w_{t}\in\R$ for all $t$. We
define
\begin{align*}
Z_{t} & =w_{t}\left(\frac{\eta_{t}}{\alpha_{t}}\left(f\left(y_{t}\right)-f\left(x^{*}\right)\right)-\frac{\eta_{t}\left(1-\alpha_{t}\right)}{\alpha_{t}}\left(f\left(y_{t-1}\right)-f\left(x^{*}\right)\right)-\frac{\eta_{t}^{2}G^{2}}{1-\beta\alpha_{t}\eta_{t}}\right)\\
 & \quad+w_{T}\left(\breg\left(x^{*},z_{t}\right)-\breg\left(x^{*},z_{t-1}\right)\right) & \forall\,1\leq t\leq T\\
S_{t} & =\sum_{i=t}^{T}Z_{i} & \forall\,1\le t\leq T+1
\end{align*}

\begin{thm}
\label{thm:acc-md-concentration-subgaussian}Suppose that $w_{t-1}\geq w_{t}+6\sigma^{2}\eta_{t}^{2}w_{t}^{2}$
for every $1\leq t\leq T$ and $\frac{w_{t}\eta_{t}^{2}}{1-\beta\alpha_{t}\eta_{t}}\leq\frac{1}{4\sigma^{2}}$
for every $0\leq t\leq T$. For every $1\leq t\leq T+1$, we have
\[
\E\left[\exp\left(S_{t}\right)\vert\F_{t}\right]\leq\exp\left(\left(w_{t-1}-w_{T}\right)\breg\left(x^{*},z_{t-1}\right)+3\sigma^{2}\sum_{i=t}^{T}w_{i}\frac{\eta_{i}^{2}}{1-\beta\alpha_{i}\eta_{i}}\right)
\]
\end{thm}

\begin{proof}
We proceed by induction on $t$. Consider the base case $t=T+1$.
We have $S_{t}=0$ and $w_{t-1}-w_{T}=0$, and the inequality follows.
Next, we consider $t\leq T$. We have
\begin{align}
\E\left[\exp\left(S_{t}\right)\vert\F_{t}\right] & =\E\left[\exp\left(Z_{t}+S_{t+1}\right)\vert\F_{t}\right]=\E\left[\E\left[\exp\left(Z_{t}+S_{t+1}\right)\vert\F_{t+1}\right]\vert\F_{t}\right]\label{eq:1-1}
\end{align}
We now analyze the inner expectation. Conditioned on $\F_{t+1}$,
$Z_{t}$ is fixed. Using the inductive hypothesis, we obtain
\begin{align}
\E\left[\exp\left(Z_{t}+S_{t+1}\right)\vert\F_{t+1}\right]\leq\exp\left(Z_{t}\right)\exp\left(\left(w_{t}-w_{T}\right)\breg\left(x^{*},z_{t}\right)+3\sigma^{2}\sum_{i=t+1}^{T}w_{i}\frac{\eta_{i}^{2}}{1-\beta\alpha_{i}\eta_{i}}\right)\label{eq:2-1}
\end{align}
Let $X_{t}=\eta_{t}\left\langle \xi_{t},x^{*}-z_{t-1}\right\rangle $.
By Lemma \ref{lem:acc-md-basic-analysis}, we have
\begin{align*}
 & \frac{\eta_{t}}{\alpha_{t}}\left(f\left(y_{t}\right)-f\left(x^{*}\right)\right)-\frac{\eta_{t}}{\alpha_{t}}\left(1-\alpha_{t}\right)\left(f\left(y_{t-1}\right)-f\left(x^{*}\right)\right)-\frac{\eta_{t}^{2}}{1-\beta\alpha_{t}\eta_{t}}G^{2}\\
 & \leq X_{t}+\frac{\eta_{t}^{2}}{\left(1-\beta\alpha_{t}\eta_{t}\right)}\left\Vert \xi_{t}\right\Vert ^{2}-\left(\breg\left(x^{*},z_{t}\right)-\breg\left(x^{*},z_{t-1}\right)\right)
\end{align*}
and thus
\begin{align*}
Z_{t} & \leq w_{t}X_{t}-\left(w_{t}-w_{T}\right)\left(\breg\left(x^{*},z_{t}\right)-\breg\left(x^{*},z_{t-1}\right)\right)+w_{t}\frac{\eta_{t}^{2}}{1-\beta\alpha_{t}\eta_{t}}\left\Vert \xi_{t}\right\Vert ^{2}
\end{align*}
Plugging into \eqref{eq:2-1}, we obtain
\begin{align*}
 & \E\left[\exp\left(Z_{t}+S_{t+1}\right)\vert\F_{t+1}\right]\\
 & \leq\exp\left(w_{t}X_{t}+\left(w_{t}-w_{T}\right)\breg\left(x^{*},z_{t-1}\right)+w_{t}\frac{\eta_{t}^{2}}{\left(1-\beta\alpha_{t}\eta_{t}\right)}\left\Vert \xi_{t}\right\Vert ^{2}+3\sigma^{2}\sum_{i=t+1}^{T}w_{i}\frac{\eta_{i}^{2}}{1-\beta\alpha_{i}\eta_{i}}\right)
\end{align*}
Plugging into \eqref{eq:1-1}, we obtain
\begin{align}
 & \E\left[\exp\left(S_{t}\right)\vert\F_{t}\right]\nonumber \\
 & \leq\exp\left(\left(w_{t}-w_{T}\right)\breg\left(x^{*},z_{t-1}\right)+3\sigma^{2}\sum_{i=t+1}^{T}w_{i}\frac{\eta_{i}^{2}}{1-\beta\alpha_{i}\eta_{i}}\right)\E\left[\exp\left(w_{t}X_{t}+w_{t}\frac{\eta_{t}^{2}}{1-\beta\alpha_{t}\eta_{t}}\left\Vert \xi_{t}\right\Vert ^{2}\right)\vert\F_{t}\right]\label{eq:3-1}
\end{align}
Next, we analyze the the expectation on the RHS of the above inequality.
We have
\begin{align}
 & \E\left[\exp\left(w_{t}X_{t}+w_{t}\frac{\eta_{t}^{2}}{1-\beta\alpha_{t}\eta_{t}}\left\Vert \xi_{t}\right\Vert ^{2}\right)\vert\F_{t}\right]\nonumber \\
 & =\E\left[\sum_{i=0}^{\infty}\frac{1}{i!}\left(w_{t}X_{t}+w_{t}\frac{\eta_{t}^{2}}{1-\beta\alpha_{t}\eta_{t}}\left\Vert \xi_{t}\right\Vert ^{2}\right)^{i}\vert\F_{t}\right]\nonumber \\
 & =\E\left[1+w_{t}\frac{\eta_{t}^{2}}{1-\beta\alpha_{t}\eta_{t}}\left\Vert \xi_{t}\right\Vert ^{2}+\sum_{i=2}^{\infty}\frac{1}{i!}\left(w_{t}X_{t}+w_{t}\frac{\eta_{t}^{2}}{1-\beta\alpha_{t}\eta_{t}}\left\Vert \xi_{t}\right\Vert ^{2}\right)^{i}\vert\F_{t}\right]\nonumber \\
 & \leq\E\left[1+w_{t}\frac{\eta_{t}^{2}}{1-\beta\alpha_{t}\eta_{t}}\left\Vert \xi_{t}\right\Vert ^{2}+\sum_{i=2}^{\infty}\frac{1}{i!}\left(w_{t}\eta_{t}\left\Vert x^{*}-z_{t-1}\right\Vert \left\Vert \xi_{t}\right\Vert +w_{t}\frac{\eta_{t}^{2}}{1-\beta\alpha_{t}\eta_{t}}\left\Vert \xi_{t}\right\Vert ^{2}\right)^{i}\vert\F_{t}\right]\nonumber \\
 & \leq\exp\left(3\left(w_{t}^{2}\eta_{t}^{2}\left\Vert x^{*}-z_{t-1}\right\Vert ^{2}+w_{t}\frac{\eta_{t}^{2}}{1-\beta\alpha_{t}\eta_{t}}\right)\sigma^{2}\right)\nonumber \\
 & \leq\exp\left(3\left(2w_{t}^{2}\eta_{t}^{2}\breg\left(x^{*},z_{t-1}\right)+w_{t}\frac{\eta_{t}^{2}}{1-\beta\alpha_{t}\eta_{t}}\right)\sigma^{2}\right)\label{eq:4-1}
\end{align}
On the first line we used the Taylor expansion of $e^{x}$, and on
the second line we used that $\E\left[X_{t}\vert\F_{t}\right]=0$.
On the third line, we used Cauchy-Schwartz and obtained
\[
X_{t}=\eta_{t}\left\langle \xi_{t},x^{*}-z_{t-1}\right\rangle \leq\eta_{t}\left\Vert \xi_{t}\right\Vert \left\Vert x^{*}-z_{t-1}\right\Vert 
\]
On the fourth line, we applied Lemma \ref{lem:helper-taylor} with
$X=\left\Vert \xi_{t}\right\Vert $, $a=w_{t}\eta_{t}\left\Vert x^{*}-z_{t-1}\right\Vert $,
and $b^{2}=w_{t}\frac{\eta_{t}^{2}}{1-\beta\alpha_{t}\eta_{t}}\le\frac{1}{4\sigma^{2}}$.
On the fifth line, we used that $\breg\left(x^{*},z_{t-1}\right)\geq\frac{1}{2}\left\Vert x^{*}-z_{t-1}\right\Vert ^{2}$,
which follows from the strong convexity of $\psi$.

Plugging in \eqref{eq:4-1} into \eqref{eq:3-1} and using that $w_{t-1}\geq w_{t}+6\sigma^{2}w_{t}^{2}\eta_{t}^{2}$,
we obtain
\[
\E\left[\exp\left(S_{t}\right)\vert\F_{t}\right]\leq\exp\left(\left(w_{t}+6\sigma^{2}w_{t}^{2}\eta_{t}^{2}-w_{T}\right)\breg\left(x^{*},z_{t-1}\right)+3\sigma^{2}\sum_{i=t}^{T}w_{i}\frac{\eta_{i}^{2}}{1-\beta\alpha_{i}\eta_{i}}\right)
\]
as needed.
\end{proof}

Theorem \ref{thm:acc-md-concentration-subgaussian} and Markov's inequality
gives us the following convergence guarantee.
\begin{cor}
\label{cor:acc-md-convergence}Suppose the sequence $\left\{ w_{t}\right\} $
satisfies the conditions of Theorem \ref{thm:acc-md-concentration-subgaussian}.
For any $\delta>0$, the following event holds with probability at
least $1-\delta$:
\begin{align*}
 & \sum_{t=1}^{T}w_{t}\left(\frac{\eta_{t}}{\alpha_{t}}\left(f\left(y_{t}\right)-f\left(x^{*}\right)\right)-\frac{\eta_{t}\left(1-\alpha_{t}\right)}{\alpha_{t}}\left(f\left(y_{t-1}\right)-f\left(x^{*}\right)\right)\right)+w_{T}\breg\left(x^{*},z_{T}\right)\\
 & \leq w_{0}\breg\left(x^{*},z_{0}\right)+\left(G^{2}+3\sigma^{2}\right)\sum_{t=1}^{T}w_{t}\frac{\eta_{t}^{2}}{1-\beta\alpha_{t}\eta_{t}}+\ln\left(\frac{1}{\delta}\right)
\end{align*}
\end{cor}

\begin{proof}
Let
\[
K=\left(w_{0}-w_{T}\right)\breg\left(x^{*},z_{0}\right)+3\sigma^{2}\sum_{t=1}^{T}w_{t}\frac{\eta_{t}^{2}}{1-\beta\alpha_{t}\eta_{t}}+\ln\left(\frac{1}{\delta}\right)
\]
By Theorem \ref{thm:acc-md-concentration-subgaussian} and Markov's
inequality, we have
\begin{align*}
\Pr\left[S_{1}\geq K\right] & \leq\Pr\left[\exp\left(S_{1}\right)\geq\exp\left(K\right)\right]\\
 & \leq\exp\left(-K\right)\E\left[\exp\left(S_{1}\right)\right]\\
 & \leq\exp\left(-K\right)\exp\left(\left(w_{0}-w_{T}\right)\breg\left(x^{*},z_{0}\right)+3\sigma^{2}\sum_{t=1}^{T}w_{t}\frac{\eta_{t}^{2}}{1-\beta\alpha_{t}\eta_{t}}\right)\\
 & =\delta
\end{align*}
Note that
\begin{align*}
S_{1} & =\sum_{t=1}^{T}Z_{t}\\
 & =\sum_{t=1}^{T}w_{t}\left(\frac{\eta_{t}}{\alpha_{t}}\left(f\left(y_{t}\right)-f\left(x^{*}\right)\right)-\frac{\eta_{t}\left(1-\alpha_{t}\right)}{\alpha_{t}}\left(f\left(y_{t-1}\right)-f\left(x^{*}\right)\right)\right)\\
 & \quad-G^{2}\sum_{t=1}^{T}w_{t}\frac{\eta_{t}^{2}}{1-\beta\alpha_{t}\eta_{t}}+w_{T}\left(\breg\left(x^{*},z_{T}\right)-\breg\left(x^{*},z_{0}\right)\right)
\end{align*}
Therefore, with probability at least $1-\delta$, we have
\begin{align*}
 & \sum_{t=1}^{T}w_{t}\left(\frac{\eta_{t}}{\alpha_{t}}\left(f\left(y_{t}\right)-f\left(x^{*}\right)\right)-\frac{\eta_{t}\left(1-\alpha_{t}\right)}{\alpha_{t}}\left(f\left(y_{t-1}\right)-f\left(x^{*}\right)\right)\right)+w_{T}\breg\left(x^{*},z_{T}\right)\\
 & \leq w_{0}\breg\left(x^{*},z_{0}\right)+\left(G^{2}+3\sigma^{2}\right)\sum_{t=1}^{T}w_{t}\frac{\eta_{t}^{2}}{1-\beta\alpha_{t}\eta_{t}}+\ln\left(\frac{1}{\delta}\right)
\end{align*}
\end{proof}

With the above result in hand, we complete the convergence analysis
by showing how to define the sequence $\left\{ w_{t}\right\} $ with
the desired properties. 
\begin{cor}
\label{cor:acc-md-convergence-final}Suppose we run the Accelerated
Stochastic Mirror Descent algorithm with the standard choices $\alpha_{t}=\frac{2}{t+1}$
and $\eta_{t}=\eta t$ with $\eta\leq\frac{1}{4\beta}$. Let $w_{T}=\frac{1}{3\sigma^{2}\eta^{2}T\left(T+1\right)\left(2T+1\right)}$
and $w_{t-1}=w_{t}+6\sigma^{2}\eta_{t}^{2}w_{t}^{2}$ for all $1\leq t\leq T$.
The sequence $\left\{ w_{t}\right\} _{0\leq t\leq T}$ satisfies the
conditions required by Corollary \ref{cor:acc-md-convergence}. By
Corollary \ref{cor:acc-md-convergence}, for any $\delta>0$, the
following events hold with probability at least $1-\delta$:
\[
f\left(y_{T}\right)-f\left(x^{*}\right)\leq O\left(\frac{\breg\left(x^{*},z_{0}\right)}{\eta T^{2}}+\left(G^{2}+\left(1+\ln\left(\frac{1}{\delta}\right)\right)\sigma^{2}\right)\eta T\right)
\]
and
\begin{align*}
\breg\left(x^{*},z_{T}\right) & \leq O\left(\breg\left(x^{*},z_{0}\right)+\left(G^{2}+\left(1+\ln\left(\frac{1}{\delta}\right)\right)\sigma^{2}\right)\eta^{2}T^{3}\right)
\end{align*}
Setting $\eta=\min\left\{ \frac{1}{4\beta},\frac{\sqrt{\breg\left(x^{*},z_{0}\right)}}{\sqrt{G^{2}+\sigma^{2}\left(1+\ln\left(\frac{1}{\delta}\right)\right)}T^{3/2}}\right\} $
to balance the two terms in the first inequality gives
\[
f\left(y_{T}\right)-f\left(x^{*}\right)\leq O\left(\frac{\beta\breg\left(x^{*},z_{0}\right)}{T^{2}}+\frac{\sqrt{\breg\left(x^{*},z_{0}\right)\left(G^{2}+\left(1+\ln\left(\frac{1}{\delta}\right)\right)\sigma^{2}\right)}}{\sqrt{T}}\right)
\]
and
\[
\breg\left(x^{*},z_{T}\right)\leq O\left(\breg\left(x^{*},z_{0}\right)\right)
\]
\end{cor}

\begin{proof}
Recall from Corollary \ref{cor:acc-md-convergence} that the sequence
$\left\{ w_{t}\right\} $ needs to satisfy the following conditions
for all $1\leq t\leq T$:
\begin{align}
w_{t}+6\sigma^{2}\eta_{t}^{2}w_{t}^{2} & \leq w_{t-1}\quad\forall1\leq t\leq T\label{eq:C1}\\
\frac{w_{t}\eta_{t}^{2}}{1-\beta\alpha_{t}\eta_{t}} & \leq\frac{1}{4\sigma^{2}}\quad\forall0\leq t\leq T\label{eq:C2}
\end{align}
We will set $\left\{ w_{t}\right\} $ so that it satisfies the following
additional condition, which will allow us to telescope the sum on
the RHS of Corollary \ref{cor:acc-md-convergence}:
\begin{equation}
w_{t-1}\frac{\eta_{t-1}}{\alpha_{t-1}}\geq w_{t}\frac{\eta_{t}\left(1-\alpha_{t}\right)}{\alpha_{t}}\quad\forall1\leq t\leq T-1\label{eq:C3}
\end{equation}
Given $w_{T}$, we set $w_{t-1}$ for every $1\leq t\leq T$ so that
the first condition \eqref{eq:C1} holds with equality:
\[
w_{t-1}=w_{t}+6\sigma^{2}\eta_{t}^{2}w_{t}^{2}=w_{t}+6\sigma^{2}\eta^{2}t^{2}w_{t}^{2}
\]
Let $C=\sigma^{2}\eta^{2}T\left(T+1\right)\left(2T+1\right)$. We
set 
\[
w_{T}=\frac{1}{C+6\sigma^{2}\eta^{2}\sum_{i=1}^{T}i^{2}}=\frac{1}{C+\sigma^{2}\eta^{2}T\left(T+1\right)\left(2T+1\right)}=\frac{1}{2\sigma^{2}\eta^{2}T\left(T+1\right)\left(2T+1\right)}
\]
Given this choice for $w_{T}$, we now verify that, for all $0\leq t\leq T$,
we have
\[
w_{t}\leq\frac{1}{C+6\sigma^{2}\eta^{2}\sum_{i=1}^{t}i^{2}}=\frac{1}{C+\sigma^{2}\eta^{2}t\left(t+1\right)\left(2t+1\right)}
\]
We proceed by induction on $t$. The base case $t=T$ follows from
the definition of $w_{T}$. Consider $t<T$. Using the definition
of $w_{t-1}$ and the inductive hypothesis, we obtain
\begin{align*}
w_{t-1} & =w_{t}+6\sigma^{2}\eta^{2}t^{2}w_{t}^{2}\\
 & \leq\frac{1}{C+6\sigma^{2}\eta^{2}\sum_{i=1}^{t}i^{2}}+\frac{6\sigma^{2}\eta^{2}t^{2}}{\left(C+6\sigma^{2}\eta^{2}\sum_{i=1}^{t}i^{2}\right)^{2}}\\
 & \leq\frac{1}{C+6\sigma^{2}\eta^{2}\sum_{i=1}^{t}i^{2}}+\frac{\left(C+6\sigma^{2}\eta^{2}\sum_{i=1}^{t}i^{2}\right)-\left(C+6\sigma^{2}\eta^{2}\sum_{i=1}^{t-1}i^{2}\right)}{\left(C+6\sigma^{2}\eta^{2}\sum_{i=1}^{t}i^{2}\right)\left(C+6\sigma^{2}\eta^{2}\sum_{i=1}^{t-1}i^{2}\right)}\\
 & =\frac{1}{C+6\sigma^{2}\eta^{2}\sum_{i=1}^{t-1}i^{2}}
\end{align*}
as needed.

Let us now verify that the second condition \eqref{eq:C2} also holds.
Using that $\frac{2t}{t+1}\leq2$, $\beta\eta\leq\frac{1}{4}$, and
$T\geq2$, we obtain
\[
\frac{w_{t}\eta_{t}^{2}}{1-\beta\alpha_{t}\eta_{t}}=\frac{w_{t}\eta^{2}t^{2}}{1-\beta\eta\frac{2t}{t+1}}\leq2w_{t}\eta^{2}t^{2}\leq\frac{2\eta^{2}t^{2}}{C}=\frac{t^{2}}{\sigma^{2}T\left(T+1\right)\left(2T+1\right)}\leq\frac{1}{\sigma^{2}\left(2T+1\right)}\leq\frac{1}{4\sigma^{2}}
\]
as needed.

Let us now verify that the third condition \eqref{eq:C3} also holds.
Since $\eta_{t}=\eta t$ and $\alpha_{t}=\frac{2}{t+1}$, we have
$\frac{\eta_{t-1}}{\alpha_{t-1}}=\frac{\eta_{t}\left(1-\alpha_{t}\right)}{\alpha_{t}}=\frac{\eta t\left(t-1\right)}{2}$.
Since $w_{t}\leq w_{t-1}$, it follows that condition \eqref{eq:C3}
holds.

We now turn our attention to the convergence. By Corollary \ref{cor:acc-md-convergence},
with probability $\geq1-\delta$, we have
\begin{align*}
 & \sum_{t=1}^{T}w_{t}\left(\frac{\eta_{t}}{\alpha_{t}}\left(f\left(y_{t}\right)-f\left(x^{*}\right)\right)-\frac{\eta_{t}\left(1-\alpha_{t}\right)}{\alpha_{t}}\left(f\left(y_{t-1}\right)-f\left(x^{*}\right)\right)\right)+w_{T}\breg\left(x^{*},z_{T}\right)\\
 & \leq w_{0}\breg\left(x^{*},z_{0}\right)+\left(G^{2}+3\sigma^{2}\right)\sum_{t=1}^{T}w_{t}\frac{\eta_{t}^{2}}{1-\beta\alpha_{t}\eta_{t}}+\ln\left(\frac{1}{\delta}\right)
\end{align*}
Grouping terms on the LHS and using that $\alpha_{1}=1$, we obtain
\begin{align*}
 & \sum_{t=1}^{T-1}\left(w_{t}\frac{\eta_{t}}{\alpha_{t}}-w_{t+1}\frac{\eta_{t+1}\left(1-\alpha_{t+1}\right)}{\alpha_{t+1}}\right)\left(f\left(y_{t}\right)-f\left(x^{*}\right)\right)+w_{T}\frac{\eta_{T}}{\alpha_{T}}\left(f\left(y_{T}\right)-f\left(x^{*}\right)\right)+w_{T}\breg\left(x^{*},z_{T}\right)\\
 & \leq w_{0}\breg\left(x^{*},z_{0}\right)+\left(G^{2}+3\sigma^{2}\right)\sum_{t=1}^{T}w_{t}\frac{\eta_{t}^{2}}{1-\beta\alpha_{t}\eta_{t}}+\ln\left(\frac{1}{\delta}\right)
\end{align*}
Since $\left\{ w_{t}\right\} $ satisfies condition \eqref{eq:C3},
the coefficient of $f\left(y_{t}\right)-f\left(x^{*}\right)$ is non-negative
and thus we can drop the above sum. We obtain
\begin{align*}
w_{T}\frac{\eta_{T}}{\alpha_{T}}\left(f\left(y_{T}\right)-f\left(x^{*}\right)\right)+w_{T}\breg\left(x^{*},z_{T}\right) & \leq w_{0}\breg\left(x^{*},z_{0}\right)+\left(G^{2}+3\sigma^{2}\right)\sum_{t=1}^{T}w_{t}\frac{\eta_{t}^{2}}{1-\beta\alpha_{t}\eta_{t}}+\ln\left(\frac{1}{\delta}\right)
\end{align*}
 Using that $w_{T}=\frac{1}{2C}$ and $w_{t}\leq\frac{1}{C}$ for
all $0\leq t\leq T-1$, we obtain
\begin{align*}
 & \frac{1}{2C}\frac{\eta_{T}}{\alpha_{T}}\left(f\left(y_{T}\right)-f\left(x^{*}\right)\right)+\frac{1}{2C}\breg\left(x^{*},z_{T}\right)\\
 & \leq\frac{1}{C}\breg\left(x^{*},z_{0}\right)+\frac{1}{C}\left(G^{2}+3\sigma^{2}\right)\sum_{t=1}^{T}\frac{\eta_{t}^{2}}{1-\beta\alpha_{t}\eta_{t}}+\ln\left(\frac{1}{\delta}\right)
\end{align*}
Thus
\begin{align*}
 & \frac{\eta_{T}}{\alpha_{T}}\left(f\left(y_{T}\right)-f\left(x^{*}\right)\right)+\breg\left(x^{*},z_{T}\right)\\
 & \leq2\breg\left(x^{*},z_{0}\right)+2\left(G^{2}+3\sigma^{2}\right)\sum_{t=1}^{T}\frac{\eta_{t}^{2}}{1-\beta\alpha_{t}\eta_{t}}+2C\ln\left(\frac{1}{\delta}\right)\\
 & =2\breg\left(x^{*},z_{0}\right)+2\left(G^{2}+3\sigma^{2}\right)\sum_{t=1}^{T}\frac{\eta_{t}^{2}}{1-\beta\alpha_{t}\eta_{t}}+2\sigma^{2}\ln\left(\frac{1}{\delta}\right)\eta^{2}T\left(T+1\right)\left(2T+1\right)
\end{align*}
Using that $\beta\eta\leq\frac{1}{4}$ and $\frac{2t}{t+1}\leq2$,
we obtain
\[
\sum_{t=1}^{T}\frac{\eta_{t}^{2}}{1-\beta\alpha_{t}\eta_{t}}=\sum_{t=1}^{T}\frac{\eta^{2}t^{2}}{1-\beta\eta\frac{2t}{t+1}}\leq\sum_{t=1}^{T}2\eta^{2}t^{2}=\frac{1}{3}\eta^{2}T\left(T+1\right)\left(2T+1\right)
\]
Plugging in and using that $\eta_{T}=\eta T$ and $\alpha_{T}=\frac{2}{T+1}$,
we obtain
\begin{align*}
 & \eta\frac{T\left(T+1\right)}{2}\left(f\left(y_{T}\right)-f\left(x^{*}\right)\right)+\breg\left(x^{*},z_{T}\right)\\
 & \leq2\breg\left(x^{*},z_{0}\right)+\left(\frac{2}{3}G^{2}+2\left(1+\ln\left(\frac{1}{\delta}\right)\right)\sigma^{2}\right)\eta^{2}T\left(T+1\right)\left(2T+1\right)\\
 & \leq2\breg\left(x^{*},z_{0}\right)+2\left(G^{2}+\left(1+\ln\left(\frac{1}{\delta}\right)\right)\sigma^{2}\right)\eta^{2}T\left(T+1\right)\left(2T+1\right)
\end{align*}
We can further simplify the bound by lower bounding $T\left(T+1\right)\geq T^{2}$
and upper bounding $T\left(T+1\right)\left(2T+1\right)\leq6T^{3}$.
We obtain
\begin{align*}
\eta T^{2}\left(f\left(y_{T}\right)-f\left(x^{*}\right)\right)+\breg\left(x^{*},z_{T}\right) & \leq4\breg\left(x^{*},z_{0}\right)+24\left(G^{2}+\left(1+\ln\left(\frac{1}{\delta}\right)\right)\sigma^{2}\right)\eta^{2}T^{3}
\end{align*}
Thus we obtain
\[
f\left(y_{T}\right)-f\left(x^{*}\right)\leq\frac{4\breg\left(x^{*},z_{0}\right)}{\eta T^{2}}+24\left(G^{2}+\left(1+\ln\left(\frac{1}{\delta}\right)\right)\sigma^{2}\right)\eta T
\]
and
\begin{align*}
\breg\left(x^{*},z_{T}\right) & \leq2\breg\left(x^{*},z_{0}\right)+12\left(G^{2}+\left(1+\ln\left(\frac{1}{\delta}\right)\right)\sigma^{2}\right)\eta^{2}T^{3}
\end{align*}
\end{proof}

\bibliographystyle{plain}
\bibliography{references}

\appendix

\section{Omitted Proofs}

\begin{proof}
\textbf{(Lemma \ref{lem:helper-taylor})} Consider two cases either
$a\ge1/(2\sigma)$ or $a\le1/(2\sigma)$. First suppose $a\ge1/(2\sigma)$.
We use the inequality $uv\le\frac{u^{2}}{4}+v^{2}$,

\begin{align*}
\E\left[1+b^{2}X^{2}+\sum_{i=2}^{\infty}\frac{1}{i!}\left(aX+b^{2}X^{2}\right)^{i}\right] & \le\E\left[1+b^{2}X^{2}+\sum_{i=2}^{\infty}\frac{1}{i!}\left(\frac{1}{4\sigma^{2}}X^{2}+a^{2}\sigma^{2}+b^{2}X^{2}\right)^{i}\right]\\
 & =\E\left[b^{2}X^{2}+\exp\left(\left(\frac{1}{4\sigma^{2}}+b^{2}\right)X^{2}+a^{2}\sigma^{2}\right)-\left(\frac{1}{4\sigma^{2}}+b^{2}\right)X^{2}-a^{2}\sigma^{2}\right]\\
 & =\E\left[\exp\left(\left(\frac{1}{4\sigma^{2}}+b^{2}\right)X^{2}+a^{2}\sigma^{2}\right)-\frac{1}{4\sigma^{2}}X^{2}-a^{2}\sigma^{2}\right]\\
 & \le\exp\left(\left(\frac{1}{4\sigma^{2}}+b^{2}\right)\sigma^{2}+a^{2}\sigma^{2}\right)\\
 & \le\exp\left(b^{2}\sigma^{2}+2a^{2}\sigma^{2}\right)
\end{align*}

Next, let $c=\max(a,b)\le1/(2\sigma)$. We have

\begin{align*}
\E\left[1+b^{2}X^{2}+\sum_{i=2}^{\infty}\frac{1}{i!}\left(aX+b^{2}X^{2}\right)^{i}\right] & =\E\left[\exp\left(aX+b^{2}X^{2}\right)-aX\right]\\
 & \le\E\left[\left(aX+\exp\left(a^{2}X^{2}\right)\right)\exp\left(b^{2}X^{2}\right)-aX\right]\\
 & =\E\left[\exp\left(\left(a^{2}+b^{2}\right)X^{2}\right)+aX\left(\exp\left(b^{2}X^{2}\right)-1\right)\right]\\
 & \le\E\left[\exp\left(\left(a^{2}+b^{2}\right)X^{2}\right)+cX\left(\exp\left(c^{2}X^{2}\right)-1\right)\right]\\
 & \le\E\left[\exp\left(\left(a^{2}+b^{2}\right)X^{2}\right)+\exp\left(2c^{2}X^{2}\right)-1\right]\\
 & \le\E\left[\exp\left(\left(a^{2}+b^{2}+2c^{2}\right)X^{2}\right)\right]\\
 & \le\exp\left(\left(a^{2}+b^{2}+2c^{2}\right)\sigma^{2}\right)
\end{align*}

In the first inequality, we use the inequality $e^{x}-x\le e^{x^{2}}\forall x$.
In the third inequality, we use $x\left(e^{x^{2}}-1\right)\le e^{2x^{2}}-1\ \forall x$.
This inequality can be proved with the Taylor expansion.

\begin{align*}
x\left(e^{x^{2}}-1\right) & =\sum_{i=1}^{\infty}\frac{1}{i!}x^{2i+1}\\
 & \le\sum_{i=1}^{\infty}\frac{1}{i!}\frac{x^{2i}+x^{2i+2}}{2}\\
 & =\frac{x^{2}}{2}+\sum_{i=2}^{\infty}\left(\frac{1+i}{2i!}\right)x^{2i}\\
 & \le\frac{x^{2}}{2}+\sum_{i=2}^{\infty}\left(\frac{2^{i}}{i!}\right)x^{2i}\\
 & \le e^{2x^{2}}-1
\end{align*}
\end{proof}

\begin{proof}
\textbf{(Lemma \eqref{lem:md-basic-analysis})} By the optimality
condition, we have
\[
\left\langle \eta_{t}\widehat{\nabla}f(x_{t})+\nabla_{x}\breg\left(x_{t+1},x_{t}\right),x^{*}-x_{t+1}\right\rangle \ge0
\]
 and thus
\[
\left\langle \eta_{t}\widehat{\nabla}f(x_{t}),x_{t+1}-x^{*}\right\rangle \leq\left\langle \nabla_{x}\breg\left(x_{t+1},x_{t}\right),x^{*}-x_{t+1}\right\rangle 
\]
Note that 
\begin{align*}
\left\langle \nabla_{x}\breg\left(x_{t+1},x_{t}\right),x^{*}-x_{t+1}\right\rangle  & =\left\langle \nabla\psi\left(x_{t+1}\right)-\nabla\psi\left(x_{t}\right),x^{*}-x_{t+1}\right\rangle \\
 & =\breg\left(x^{*},x_{t}\right)-\breg\left(x_{t+1},x_{t}\right)-\breg\left(x^{*},x_{t+1}\right)
\end{align*}
 and thus
\begin{align*}
\eta_{t}\left\langle \widehat{\nabla}f(x_{t}),x_{t+1}-x^{*}\right\rangle  & \leq\breg\left(x^{*},x_{t}\right)-\breg\left(x^{*},x_{t+1}\right)-\breg\left(x_{t+1},x_{t}\right)\\
 & \leq\breg\left(x^{*},x_{t}\right)-\breg\left(x^{*},x_{t+1}\right)-\frac{1}{2}\left\Vert x_{t+1}-x_{t}\right\Vert ^{2}
\end{align*}
 where we have used that $\breg\left(x_{t+1},x_{t}\right)\geq\frac{1}{2}\left\Vert x_{t+1}-x_{t}\right\Vert ^{2}$
by the strong convexity of $\psi$.

By convexity,

\[
f\left(x_{t}\right)-f\left(x^{*}\right)\le\left\langle \nabla f\left(x_{t}\right),x_{t}-x^{*}\right\rangle =\left\langle \xi_{t},x^{*}-x_{t}\right\rangle +\left\langle \widehat{\nabla}f\left(x_{t}\right),x_{t}-x^{*}\right\rangle 
\]
Combining the two inequalities, we obtain
\begin{align*}
 & \eta_{t}\left(f\left(x_{t}\right)-f\left(x^{*}\right)\right)+\breg\left(x^{*},x_{t+1}\right)-\breg\left(x^{*},x_{t}\right)\\
 & \le\eta_{t}\left\langle \xi_{t},x^{*}-x_{t}\right\rangle +\eta_{t}\left\langle \widehat{\nabla}f(x_{t}),x_{t}-x_{t+1}\right\rangle -\frac{1}{2}\left\Vert x_{t+1}-x_{t}\right\Vert ^{2}\\
 & \le\eta_{t}\left\langle \xi_{t},x^{*}-x_{t}\right\rangle +\frac{\eta_{t}^{2}}{2}\left\Vert \widehat{\nabla}f(x_{t})\right\Vert ^{2}
\end{align*}
Using the triangle inequality and the bounded gradient assumption
$\left\Vert \nabla f(x)\right\Vert \leq G$ , we obtain
\[
\left\Vert \widehat{\nabla}f(x_{t})\right\Vert ^{2}=\left\Vert \xi_{t}+\nabla f(x_{t})\right\Vert ^{2}\leq2\left\Vert \xi_{t}\right\Vert ^{2}+2\left\Vert \nabla f(x_{t})\right\Vert ^{2}\leq2\left(\left\Vert \xi_{t}\right\Vert ^{2}+G^{2}\right)
\]
Thus
\[
\eta_{t}\left(f\left(x_{t}\right)-f\left(x^{*}\right)\right)+\breg\left(x^{*},x_{t+1}\right)-\breg\left(x^{*},x_{t}\right)\leq\eta_{t}\left\langle \xi_{t},x^{*}-x_{t}\right\rangle +\eta_{t}^{2}\left(\left\Vert \xi_{t}\right\Vert ^{2}+G^{2}\right)
\]
as needed.
\end{proof}

\begin{proof}
\textbf{(Lemma \ref{lem:acc-md-basic-analysis})} Starting with smoothness,
we obtain
\begin{align*}
f\left(y_{t}\right) & \le f\left(x_{t}\right)+\left\langle \nabla f\left(x_{t}\right),y_{t}-x_{t}\right\rangle +G\left\Vert y_{t}-x_{t}\right\Vert +\frac{\beta}{2}\left\Vert y_{t}-x_{t}\right\Vert ^{2}\ \forall x\in\dom\\
 & =f\left(x_{t}\right)+\left\langle \nabla f\left(x_{t}\right),y_{t-1}-x_{t}\right\rangle +\left\langle \nabla f\left(x_{t}\right),y_{t}-y_{t-1}\right\rangle +G\left\Vert y_{t}-x_{t}\right\Vert +\frac{\beta}{2}\left\Vert y_{t}-x_{t}\right\Vert ^{2}\\
 & =\left(1-\alpha_{t}\right)\underbrace{\left(f\left(x_{t}\right)+\left\langle \nabla f\left(x_{t}\right),y_{t-1}-x_{t}\right\rangle \right)}_{\text{convexity}}+\alpha_{t}\underbrace{\left(f\left(x_{t}\right)+\left\langle \nabla f\left(x_{t}\right),y_{t-1}-x_{t}\right\rangle \right)}_{\text{convexity}}\\
 & +\alpha_{t}\left\langle \nabla f\left(x_{t}\right),z_{t}-y_{t-1}\right\rangle +G\left\Vert y_{t}-x_{t}\right\Vert +\frac{\beta}{2}\left\Vert y_{t}-x_{t}\right\Vert ^{2}\\
 & \le\left(1-\alpha_{t}\right)f\left(y_{t-1}\right)+\alpha_{t}f\left(x_{t}\right)+\alpha_{t}\left\langle \nabla f\left(x_{t}\right),z_{t}-x_{t}\right\rangle +G\underbrace{\left\Vert y_{t}-x_{t}\right\Vert }_{=\alpha_{t}\left\Vert z_{t}-z_{t-1}\right\Vert }+\frac{\beta}{2}\underbrace{\left\Vert y_{t}-x_{t}\right\Vert ^{2}}_{=\alpha_{t}^{2}\left\Vert z_{t}-z_{t-1}\right\Vert ^{2}}\\
 & =\left(1-\alpha_{t}\right)f\left(y_{t-1}\right)+\alpha_{t}f\left(x_{t}\right)+\alpha_{t}\left\langle \nabla f\left(x_{t}\right),z_{t}-x_{t}\right\rangle +G\alpha_{t}\left\Vert z_{t}-z_{t-1}\right\Vert +\frac{\beta}{2}\alpha_{t}^{2}\left\Vert z_{t}-z_{t-1}\right\Vert ^{2}
\end{align*}
 By the optimality condition for $z_{t}$,
\[
\eta_{t}\left\langle \widehat{\nabla}f(x_{t}),z_{t}-x^{*}\right\rangle \leq\left\langle \nabla_{x}\breg\left(z_{t},z_{t-1}\right),x^{*}-z_{t}\right\rangle =\breg\left(x^{*},z_{t-1}\right)-\breg\left(z_{t},z_{t-1}\right)-\breg\left(x^{*},z_{t}\right)
\]
Rearranging, we obtain
\begin{align*}
\breg\left(x^{*},z_{t}\right)-\breg\left(x^{*},z_{t-1}\right)+\breg\left(z_{t},z_{t-1}\right) & \leq\eta_{t}\left\langle \widehat{\nabla}f\left(x_{t}\right),x^{*}-z_{t}\right\rangle =\eta_{t}\left\langle \nabla f\left(x_{t}\right)+\xi_{t},x^{*}-z_{t}\right\rangle 
\end{align*}
By combining the two inequalities, we obtain
\begin{align*}
 & f\left(y_{t}\right)+\frac{\alpha_{t}}{\eta_{t}}\left(\breg\left(x^{*},z_{t}\right)-\breg\left(x^{*},z_{t-1}\right)+\breg\left(z_{t},z_{t-1}\right)\right)\\
 & \leq\left(1-\alpha_{t}\right)f\left(y_{t-1}\right)+\alpha_{t}\underbrace{\left(f\left(x_{t}\right)+\left\langle \nabla f\left(x_{t}\right),x^{*}-x_{t}\right\rangle \right)}_{\text{convexity}}\\
 & +G\alpha_{t}\left\Vert z_{t}-z_{t-1}\right\Vert +\frac{\beta}{2}\alpha_{t}^{2}\left\Vert z_{t}-z_{t-1}\right\Vert ^{2}+\alpha_{t}\left\langle \xi_{t},x^{*}-z_{t}\right\rangle \\
 & \leq\left(1-\alpha_{t}\right)f\left(y_{t-1}\right)+\alpha_{t}f\left(x^{*}\right)+G\alpha_{t}\left\Vert z_{t}-z_{t-1}\right\Vert +\frac{\beta}{2}\alpha_{t}^{2}\left\Vert z_{t}-z_{t-1}\right\Vert ^{2}+\alpha_{t}\left\langle \xi_{t},x^{*}-z_{t}\right\rangle 
\end{align*}
Subtracting $f\left(x^{*}\right)$ from both sides, rearranging, and
using that $\breg\left(z_{t},z_{t-1}\right)\geq\frac{1}{2}\left\Vert z_{t}-z_{t-1}\right\Vert ^{2}$,
we obtain
\begin{align*}
 & f\left(y_{t}\right)-f\left(x^{*}\right)+\frac{\alpha_{t}}{\eta_{t}}\left(\breg\left(x^{*},z_{t}\right)-\breg\left(x^{*},z_{t-1}\right)\right)\\
 & \leq\left(1-\alpha_{t}\right)\left(f\left(y_{t-1}\right)-f\left(x^{*}\right)\right)+\alpha_{t}\left\langle \xi_{t},x^{*}-z_{t}\right\rangle +G\alpha_{t}\left\Vert z_{t}-z_{t-1}\right\Vert -\alpha_{t}\frac{1-\beta\alpha_{t}\eta_{t}}{2\eta_{t}}\left\Vert z_{t}-z_{t-1}\right\Vert ^{2}\\
 & =\left(1-\alpha_{t}\right)\left(f\left(y_{t-1}\right)-f\left(x^{*}\right)\right)+\alpha_{t}\left\langle \xi_{t},x^{*}-z_{t-1}\right\rangle +\alpha_{t}\left\langle \xi_{t},z_{t}-z_{t-1}\right\rangle +G\alpha_{t}\left\Vert z_{t}-z_{t-1}\right\Vert -\alpha_{t}\frac{1-\beta\alpha_{t}\eta_{t}}{2\eta_{t}}\left\Vert z_{t}-z_{t-1}\right\Vert ^{2}\\
 & \le\left(1-\alpha_{t}\right)\left(f\left(y_{t-1}\right)-f\left(x^{*}\right)\right)+\alpha_{t}\left\langle \xi_{t},x^{*}-z_{t-1}\right\rangle +\alpha_{t}\left\Vert z_{t}-z_{t-1}\right\Vert \left(\left\Vert \xi_{t}\right\Vert +G\right)-\alpha_{t}\frac{1-\beta\alpha_{t}\eta_{t}}{2\eta_{t}}\left\Vert z_{t}-z_{t-1}\right\Vert ^{2}\\
 & \leq\left(1-\alpha_{t}\right)\left(f\left(y_{t-1}\right)-f\left(x^{*}\right)\right)+\alpha_{t}\left\langle \xi_{t},x^{*}-z_{t-1}\right\rangle +\frac{\alpha_{t}\eta_{t}}{2\left(1-\beta\alpha_{t}\eta_{t}\right)}\left(\left\Vert \xi_{t}\right\Vert +G\right)^{2}
\end{align*}
Finally, we divide by $\frac{\alpha_{t}}{\eta_{t}}$, and obtain
\begin{align*}
 & \frac{\eta_{t}}{\alpha_{t}}\left(f\left(y_{t}\right)-f\left(x^{*}\right)\right)+\breg\left(x^{*},z_{t}\right)-\breg\left(x^{*},z_{t-1}\right)\\
 & \leq\frac{\eta_{t}}{\alpha_{t}}\left(1-\alpha_{t}\right)\left(f\left(y_{t-1}\right)-f\left(x^{*}\right)\right)+\eta_{t}\left\langle \xi_{t},x^{*}-z_{t-1}\right\rangle +\frac{\eta_{t}^{2}}{2\left(1-\beta\alpha_{t}\eta_{t}\right)}\left(\left\Vert \xi_{t}\right\Vert +G\right)^{2}\\
 & \leq\frac{\eta_{t}}{\alpha_{t}}\left(1-\alpha_{t}\right)\left(f\left(y_{t-1}\right)-f\left(x^{*}\right)\right)+\eta_{t}\left\langle \xi_{t},x^{*}-z_{t-1}\right\rangle +\frac{\eta_{t}^{2}}{1-\beta\alpha_{t}\eta_{t}}\left(\left\Vert \xi_{t}\right\Vert ^{2}+G^{2}\right)
\end{align*}
\end{proof}

\end{document}